\documentclass[reqno,11pt]{amsart}
\usepackage{amssymb}
\usepackage{fancyhdr}
\usepackage[active]{srcltx}
\usepackage{amsmath,amssymb}
\usepackage[T1]{fontenc}  
\usepackage{setspace}                   
\usepackage{mathdots}
\usepackage{indentfirst}   
\usepackage{courier}     
\usepackage{type1cm}                    
\usepackage{listings}                   
\usepackage{titletoc}
\usepackage{amsmath,amscd,amsthm}
\usepackage{amsfonts}   
\usepackage[all]{xy}
\usepackage{color}
\newtheorem{thm}{Theorem}                                          
\newtheorem{prop}[thm]{Proposition}                          
\newtheorem{lem}[thm]{Lemma}
\newtheorem{cor}[thm]{Corollary}

\newtheorem{definition}[thm]{Definition}
\newtheorem{remark}{Remark}[section]

\def\FF{{\mathcal F}}
\def\GG{{\mathcal G}}

\def\QQ{{\mathcal Q}}
\def\XX{{\mathcal X}}

\def\RR{{\mathcal R}}
\def\PP{{\mathcal P}}

\def\fd{{\, \rightarrow\, }}

\def\lfd{{\, \longrightarrow\, }}

\def\nada {{\, \rule{0.4cm}{0.5pt}\,}}

\def\Hom {\mathop {\rm Hom} \nolimits}
\def\Ext {\mathop {\rm Ext} \nolimits}

\def\End {\mathop {\rm End}\nolimits}

\def\supp {\mathop {\rm supp}\nolimits}
\def\Supp {\mathop {\rm Supp}\nolimits}
\def\mod {\mathop {\rm mod}\nolimits}

\makeatletter
\@addtoreset{equation}{section}
\makeatother

\renewcommand\thefigure{\thesection.\@arabic\c@figure}
\renewcommand\thetable{\thesection.\@arabic\c@table}

\subjclass[2010]{16G10,16G70}
\keywords{Stratifying System, Hereditary Algebra, Euclidean Algebra, Exceptional Sequence} 

\begin{document}
\setcounter{page}{1} 
\label{firstpage}
\title{Stratifying systems over the hereditary path algebra with quiver $\mathbb{A}_{p,q}$}
\author{Paula  Cadavid}
\address[Paula A. Cadavid]{USP, Instituto de Matemática e Estatística,
Rua do Matão 1010, Cidade Universitaria, São Paulo-SP}
\email{paca@ime.usp.br}

\author{Eduardo do N. Marcos}
\address[Eduardo do N. Marcos]{USP, Instituto de Matemática e Estatística,
Rua do Matão 1010, Cidade Universitaria, São Paulo-SP}
\email{enmarcos@ime.usp.br}

\begin{abstract}
The authors have proved in [J. Algebra Appl. 14 (2015), no. 6]  that the size 
of a stratifying system  over a finite-dimensional hereditary path algebra $A$ is at most $n$, where $n$ is the number of isomorphism classes of simple 
$A$-modules. Moreover, if $A$ is of Euclidean type a stra\-ti\-fying  system over $A$ has at most $n-2$ regular modules.  In this work, we construct a family of  stra\-ti\-fying systems  of size $n$ with a maximal number of regular elements,  over the hereditary  path algebra with quiver $\widetilde{\mathbb {A}}_{p,q} $, canonically oriented.
\end{abstract}

\maketitle


\section{Preliminaries}

Throughout this article the term algebra  means associative and  finite dimensional basic algebra over an algebraically closed field $K$. Therefore any algebra can be viewed as a quotient $A=KQ/I$ of a path algebra $KQ$, where $Q$ is a quiver and $I$ is an admissible ideal of $KQ$. For a vertex $v$ of $Q$, $S_v$ will denote the simple module associated with $v$. The projective cover and the injective envelope of $S_v$  will be denoted by $P_v$ and $I_v$, respectively.

 If $A$ is a $K$-algebra then  $\mod A$ denotes the category whose objects are finitely generated right $A$-modules, $D: \mod A \lfd \mod A^{op}$ the usual duality $\Hom_K(\nada, K)$, $\tau$ the Auslander-Reiten translation, and  $\Gamma(\mod A)$ the Auslander-Reiten quiver of $\mod A$.

The concept of stratifying system ($s.s.$) was introduced as a  generalization of the standard 
modules by Erdmann and Sáenz in \cite{ES}. Later Marcos et al. introduced in \cite{MMS2, MMS1} 
 $s.s.$ via relative simple modules and via relative projective modules.
So there are, in the literature, various equivalent definitions of stratifying sistems.
In this work we use the following definition.

\begin{definition} \cite{MMS2}\label{ss} A {\bf stratifying system} ($s.s.$) of size $t$ consists
  of a completely ordered set  $X = \{X_i\}_{i=1}^{t}$ of indecomposable $A$-modules
  satisfying the following conditions: 
\begin{enumerate}
\item $\Hom_{A}(X_j,X_i)=0$, for $j > i$.
\item $\Ext_{A}^{1}(X_j, X_i)=0,$ for $j \geq i$ .   
\end{enumerate}
\end{definition}

Let $A$ be a hereditary $K$-algebra. An $A$-module $M$ is called {\bf exceptional}
if $\End_A(M)\cong K$ and $\Ext^{1}_{A}(M,M)=0$ and a sequence $X=(X_1, \ldots,X_t)$  is called {\bf exceptional} if  it is a sequence of exceptional modules satisfying $\Hom_A(X_i,X_j)=0$, for $i>j$ and $\Ext^{1}_A(X_i,X_j)=0$ for $i \geq j.$ A sequence $X=(X_1,\ldots,X_n)$  is said to be {\bf complete} if $n$ is the number of isomorphism classes of simple 
$A$-modules.

\begin{lem}\label{excseq} If $A$ is a hereditary $K$-algebra, then the sequence of $A$-modules $(X_1,\ldots, X_t)$ 
is an exceptional sequence if and only if, the set $X = \{X_i\}_{i=1}^{t}$ is a $s.s.$ over $A$. 
\end{lem}

\begin{proof}
 The only fact that needs to be observed is that if an indecomposable module does not have self extensions then its endomorphism ring is $K$.
 This is a consequence of \cite[Chapter VIII, Lemma 3.3]{ASS1}, which states that if $A$ is a finite dimensional hereditary algebra over an algebraically closed
 field $K$ and $T_1$ and $T_2$ are two indecomposable $A$ modules with $\Ext^1_A(T_1, T_2) = 0$ then every nonzero morphism between $T_1$ and $T_2$ is a monomorphism or 
 an epimorphism. In particular if $T=T_1 =T_2$ then all nonzero morphism is an isomorphism, so $\End_A(T) \cong K$.\end{proof}
 
 Since all the algebras in this work are hereditary, the former lemma gives a justification for identifying  
a $s.s.$ $X = \{X_i\}_{i=1}^{t}$  with the exceptional sequence $(X_1,\ldots, X_t)$. 

We now stating some known statements, which we will use.

\begin{lem}\cite[Lemma 1]{CB} \label{estender}Any exceptional sequence $(X_{1},\ldots,X_{a},Z_{1},\ldots,Z_{c})$
can be enlarged to a complete 
sequence $(X_{1},\ldots, X_{a},Y_{1}, \ldots, Y_{b},Z_{1},\ldots, Z_{c}).$
\end{lem}
\begin{lem}\cite[Lemma 2]{CB} \label{unico} If $(X_{1},\dots,X_{n})$ and $(Y_{1},\dots,Y_n)$ are complete 
exceptional sequences such that  $X_j\cong Y_j$ for $j \not= i$, then $X_i \cong Y_i$.
\end{lem}

The authors of this work, proved in \cite{CMa}, Lemma 3.1,  that if $A \cong KQ$ is a hereditary algebra then the size 
of a $s.s.$ over $A$ is at most $n=|Q_0|$. Moreover, if $A$ is a hereditary algebra of Euclidean type  
and $X$ is a $s.s.$ of regular modules then the  size  of  $X$ is at most $n-2$. 

As in the case of  exceptional sequences, we say that a  $s.s.$  $X$ is complete ($c.s.s.$) if the size of $X$ is $n$.


The algebra $K(\widetilde{\mathbb {A}}_{p,q}) $ will denote the hereditary algebra whose quiver is of type 
$\widetilde{\mathbb {A}}_{p,q}$, canonically 
oriented, that is it has $p$ consecutive arrows in one direction and $q$ consecutive arrows in the other direction, we will always assume, without loss of generality, that it has $p$ consecutive arrows counterclockwise and $q$ consecutive arrows clockwise. 


The main objective of this article is to construct $c.s.s.$ with maximal number of regular elements 
over $K(\widetilde{\mathbb {A}}_{p,q}) $.

\section{The algebra $K(\widetilde{\mathbb{A}}_{p,q})$}

In this section  $p$ and $q$ are integers such that $1\leq p\leq q$. Let $\widetilde{\mathbb{A}}_{p,q}$
be the canonically  oriented  Euclidean quiver, whose picture is the following:

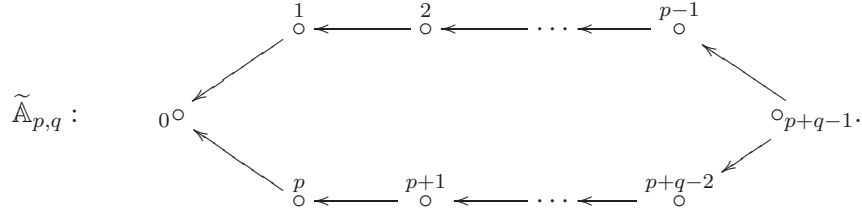
\begin{displaymath}
\xymatrix @!0 @R=3em @C=4pc{
& &  \ar[dl]  \stackrel{1}{\circ}& \stackrel{2}{\circ} \ar[l] & \ar[l]  \cdots & \ar[l] \stackrel{p-1}{\circ} \\
\widetilde{\mathbb{A}}_{p,q}: & _{0}\circ    &           &       &     &&  \ar[ul] \ar[dl] \,\,\,\,\circ_{p+q-1}.\\
& &  \ar[ul]  \stackrel{p}{\circ}&   \stackrel{p+1}{\circ}\ar[l] & \ar[l]  \cdots &  \stackrel{p+q-2}{\circ}\ar[l] 
  }
\end{displaymath}

For any non-negative integers $m,n$ we set by definition 
\begin{displaymath}
                   \begin {array}{rll}
                        [m,n]: = &\left \{
                         \begin{array}{ll}
                              \{m,m+1, \ldots,n\} \textrm{ if } m\leq n \\
                               \emptyset , \textrm{ if  } m>n.\\
                         \end{array} \right.\\
                  \end{array} 
           \end{displaymath}

We give a  construction of $s.s.$ over $A =K\widetilde{\mathbb {A}}_{p,q}$ with a maximal number regular mo\-dules. 
The  Auslander-Reiten quiver of such algebra, has a preprojective component, a preinjective component, an infinite family
of homogeneous tubes, and 2 connected regular components, that are tubes but not homogeneous, which are denoted by $\mathcal{T}_{\infty}$ and $\mathcal{T}_{0}$ (see Theorem \ref{desc}).

Observe that if an indecomposable module $T$ is in a homogeneous tube then $\tau(T)=T$, therefore it has self 
extension, that is $\Ext_A^1(T, T) \neq 0$.

We give the description of the non-homogeneous tubes.
 
\begin{enumerate}
\item[(a)] The simple regular representations in the tube $\mathcal{T}_{\infty}$ are the simple modules
$E_{i}^{(\infty)}=S_{i}, i \in [1,p-1],$ and 
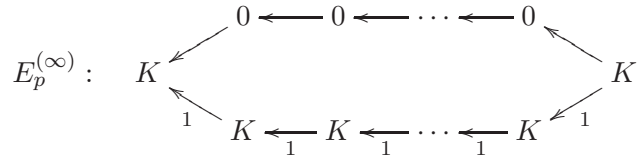
\[\xymatrix @!0@R=2em @C=3pc{
                                                       &&  \ar[dl] 0 & 0 \ar[l] & \ar[l]  \cdots & \ar[l] 0 \\
E_{p}^{(\infty)}:& K &  & & & &  \ar[ul] \ar[dl]^{1} K\\
                                                      && \ar[ul]^{1} K & K \ar[l]^{1} & \ar[l]^{1}  \cdots & K \ar[l]^{1}
  }
	\]
\begin{equation}\label{propriedades1} \text{ satisfying }\tau E_{i+1}^{(\infty)}=E_{i}^{(\infty)},\text{ for } i \in [1, p-1], 
\text{ and  } \tau E_{1}^{(\infty)}=E_{p}^{(\infty)}.
\end{equation}
  
\item[(b)] The simple regular representations in the tube $ \mathcal{T}_{0}$ are the simple modules
$E_{j}^{(0)}=S_{p+j-1},$ $j \in [1,q-1]$ and$$
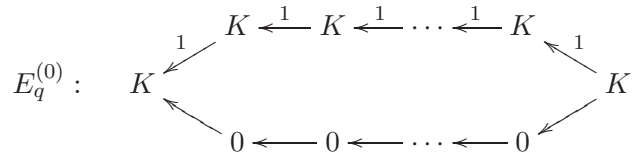
\xymatrix @!0 @R=2em @C=3pc {
 &&   \ar[dl]_{1} K & K \ar[l]_{1} & \ar[l]_{1}  \cdots & \ar[l]_{1} K \\
  E_{q}^{(0)} :&K &  &  & & &  \ar[ul]_{1} \ar[dl] K\\
 & &  \ar[ul] 0 & 0 \ar[l] & \ar[l] \cdots & 0 \ar[l]\\
  }$$
 \begin{equation}\label{propriedades2}  \text{ satisfying }
\tau E_{j+1}^{0}=E_{j}^{0},\text{ for }  j \in [1,q-1],\text{ and }\tau E_{1}^{0}=E_{q}^{0}.
\end{equation}

\item[(c)] The simple regular representation in the tube  $\mathcal{T}_{\lambda}$, 
with $\lambda \in K \setminus \{0\}$ is
\begin{displaymath}
\xymatrix @!0 @R=2em @C=3pc {
  & &  \ar[dl]_{\lambda} K & K \ar[l]_{1} & \ar[l]_{1}  \cdots & \ar[l]_{1} K \\
 E^{(\lambda)} :&K &  &  & & &  \ar[ul]_{1} \ar[dl]_{1} K\\
 & &  \ar[ul]_{1} K & K \ar[l]_{1} & \ar[l]_{1} \cdots & K \ar[l]_{1}\\
  }
\end{displaymath}
\end{enumerate}

The following theorem  give us information about the tubular components of $\Gamma (\mod K\widetilde{\mathbb {A}}_{p,q})$.

\begin{thm}\label{desc}\cite[Chap. XIII, Theorem 2.5]{SS2} Let $A=K\widetilde{\mathbb {A}}_{p,q}$. Then
every component in the regular  part $\RR(A)$ of $\Gamma(\mod A)$ is one of the following stable tubes:
\begin{enumerate}
\item The tube  $\mathcal{T}_{\infty}$ of rank $p$ containing
 $E_{1}^{(\infty)},\ldots,E_{p}^{(\infty)}$,
\item The tube $\mathcal{T}_{0}$ of rank $q$
containing   $E_{1}^{(0)}, \ldots, E_{q}^{(0)},$
\item  The tube  $\mathcal{T}_{\lambda}$ of rank $1$
containing  $E^{(\lambda)}$, with $\lambda\in K\setminus \{0\}.$
\end{enumerate}
Where  $E_{j}^{(\infty)}$, $ E_{i}^{(0)}$ and  $E^{(\lambda)}$ are the simple regular  $A$-modules previously defined.
\end{thm}
We use the simple regular modules in the tubes $\mathcal{T}_{\infty}$ and
$\mathcal{T}_{0}$ to build a $s.s.$ of  size $p+q-2$ over $K\widetilde{\mathbb {A}}_{p,q}$.
Let $F_i$ denote $E_{p-i}^{(\infty)}$, for $i \in [1,p-1]$, and let $G_i$ denote $E_{q-i}^{(0)}$, for $i \in[1,q-1]$. 
Then we have that
\begin{equation}\label{propriedades3}
\tau F_i=F_{i+1}, \text{ for }  i\in [1, p-2], \tau F_{p-1}=E_{p}^{(\infty)} \text{ and } \tau E_p= F_1 ,
\end{equation}
\begin{equation}\label{propriedades4}
\tau G_i= G_{i+1}, \text{ for  } i\in [1, q-2], \tau E_q=G_{1} \text{ and }  \tau E_q=G_{1}.
\end{equation}
We set, $\FF:= \{F_1, \ldots,F_{p-1}\}$, $F:= (F_1, \ldots,F_{p-1})$, $\GG:=\{G_1, \ldots,G_{q-1}\}$ and $G:=(G_1, \ldots,G_{q-1}).$ 

\begin{remark}
 We remark that $\FF$ is empty  if $p=1$, also $\GG$ is empty if $q=1$. In the case when $p=1$ and $q=1$ our algebra is the Kronecker algebra whose stratifying systems where described in \cite{CMa}. So we assume from now on that $q>1$.
\end{remark}

\begin{prop} Let $\FF$ and $\GG$ be the sets of indecomposable modules defined above. Then $\FF$ and $\GG$ are  $s.s.$, unless $p=1$ in which
case $\FF$ is empty.
\end{prop}
 
\begin{proof}
This is a consequence of (\ref{propriedades3}), (\ref{propriedades4}) and the Auslander-Reiten formula.
\end{proof}

Since Lemma \ref{excseq} allows us to identify a $s.s.$ with an exceptional sequence, using  the previous proposition we can see that
 $F$ and $G$ are $s.s.$

We denote by $(F, G)$ the sequence where the elements of $F$ appear before the elements of $G$. Note that as the elements 
of $F$ and $G$ are in different  orthogonal tubes (see  Theorem \ref{desc}) then $(F,G)$ is a $s.s.$

Let $\{S_i\}_{i=1}^{n}$ be a complete  set of  the isomorphism classes of simple $A$-modules. The support of an $A$-module $M$ is the set 
 $$\supp\,M := \{i \in [1,n]|[M:S_i] \neq 0\},$$ 
 where $ [M:S_i] $ is the number of composition factors of M that are isomorphic to $S_i$, and the support  of a set of $A$-modules $\XX$ is  $$\Supp\, \XX:= \cup_{M \in \XX}\supp\,M.$$
If the elements of $\XX$ are regular modules,  $ i \in \mathbb{Z}\setminus\{0\}$ we denote  by $\tau^{i}(\XX)$  the set 
$$\tau^{i}(\XX):=\{\tau^{i} M| M \in \XX  \}.$$

Observing the behavior of the indecomposable modules in the standard tubes above, we state the following lemma whose proof we leave to the reader.

\begin{lem} \label{potencies}Let $\FF$ and $\GG$ be the sets of indecomposable modules defined above, with $p>1$. Then 
\begin{enumerate}
\item
      \begin{enumerate} 
      \item $\tau^i(\FF)=(\FF \setminus\{F_{i}\}) \cup \{E_{p}^{(\infty)}\}$, for $i \in [1,p-1].$
      \item $\tau^{-i}(\FF)=(\FF\setminus\{F_{p-i}\})\cup \{E_{p}^{(\infty)}\}$, for $ i \in [1,p-1].$     
      \item $\tau^p(\FF)= \tau^{-p}(\FF)=\FF.$
      \item For each $n \in \mathbb{N}$, $n \not = 0$, we have that 
            \begin{displaymath}
                   \begin {array}{rll}
                   \tau^n(\FF) = &\left \{ 
                          \begin {array}{ll}
                          \FF, \textrm{ if } p\mid n\\
                          \tau^{r}(\FF),\textrm{ if }n\equiv r\,\,(\mod\,p),\textrm{ with } r \in [1,p-1]. \\
                          \end{array} \right.\\
                   \end{array}
            \end{displaymath}

            \begin{displaymath}
                   \begin {array}{rll}
                         \tau^{-n}(\FF) = &\left \{
                         \begin{array}{ll}
                               \FF, \textrm{ if } p|n\\
                               \tau^{-r}(\FF),\textrm{ if } n\equiv r \,\,(\mod \,p), \textrm{ with } r \in [1,p-1].\\
                         \end{array} \right.\\
                  \end{array} 
           \end{displaymath}
     \end{enumerate}
\item
\begin{enumerate}
\item $\tau^i(\GG)=(\GG \setminus \{G_{i}\}) \cup \{E_{q}^{(0)}\},\textrm{ for }  i \in  [1, q-1]$.
\item $\tau^{-i}(\GG)=(\GG \setminus \{G_{q-i}\}) \cup \{E_{q}^{(0)}\},\textrm{ for } i \in  [1, q-1]$.
\item $\tau^q(\GG)= \tau^{-q}(\GG)=\GG.$

\item For each $n \in \mathbb{N} \setminus \{0\}$ we have that

\begin{displaymath}
\begin {array}{rll}
\tau^n(\GG) = &\left \{ \begin {array}{ll}
\GG, \textrm{ if } q\mid n\\
\tau^{r}(\GG),\textrm{ if } n\equiv r \,\, (\mod \,q ), \textrm{with } r \in [1, q-1]\\
\end{array} \right.\\
\end{array}
\end{displaymath}

\begin{displaymath}
\begin {array}{rll}
\tau^{-n}(\GG) = &\left \{ \begin {array}{ll}
\GG, \textrm{ if } q\mid n\\
\tau^{-r}(\GG),\textrm{ if } n\equiv r \,\,(\mod\,q), \textrm{ with } r \in [1, q-1] \\
\end{array} \right.\\
\end{array}
\end{displaymath}
\end{enumerate}
\end{enumerate}
\end{lem}

The lemma above give us the following corollary.

\begin{cor}\label{suporte}Let  $n,r, p\in \mathbb{N}, p> 1$  and $n\not =0$. Then:
\begin{enumerate}
\item 
      \begin{enumerate}
       \item If $p\mid n$, then $\Supp \tau^n \FF = [1, p-1]=\Supp \FF.$
       \item If $p \nmid n$  and $r \in [1,p-1]$ is such that $n\equiv r \,\,(\mod \,p)$, then 
           \[ \begin{array}{lcl} 
           \Supp \tau^{n}\FF =\tau^{r}\FF&=&[0, p+q-1] \setminus \{p-r\},\\
           \Supp \tau^{-n}\FF= \tau^{-r}\FF  &=& [0,p+q-1] \setminus \{r\}.\\
           \end{array}\]
       \end{enumerate}
\item  
       \begin{enumerate}
       \item If $q\mid n$, then $\Supp \tau^n \GG = [p, p+q-2] = \GG.$
       \item If  $q \nmid n$ and $r \in [1,q-1]$ is such that $n\equiv r \,\,(\mod \,q)$, then 
        \[  \begin{array}{rcl}    
        \Supp \tau^{n}\GG =\tau^{r}\GG&=& [0, p+q-1 ] \setminus \{p+q-r-1\},\\
        \Supp \tau^{-n}\GG=\tau^{-r}\GG &=& [ 0,p+q-1] \setminus \{p+r-1\}.\\
      \end{array}\]
     \end{enumerate}
 \end{enumerate}
 \end{cor}
 
 The following lemma will be used. 
 
\begin{lem}\label{repete} Let $A$ be a hereditary algebra. Then 
\begin{enumerate} 
\item If $P_j$ and $P_m$ are  projective indecomposable $A$-modules then
$$\Ext^1(\tau^{-t} P_j, \tau^{-t-r} P_m)=0,\text{ for all } t, \,r\geq 0.$$
\item If $ I_j$ and $I_m$ are injective indecomposable $A$-modules then 
     $$\Ext^1(\tau^t I_j, \ \tau^{t-r} I_m)=0,\text{ for all } t\geq r\geq 0.$$
\end{enumerate}
\end{lem}
      
\begin{proof} We just prove the statement (1). The proof of (2) is similar. By the Auslander-Reiten
formula
$\Ext^{1}_A(\tau^{-t}P_{j},\tau^{-t-r}P_{m}) \cong  D\Hom_A(\tau^{-r-1}P_{m},P_{j}).$
But $\Hom_A(\tau^{-r-1}P_{m},P_{j})=0,$  otherwise $P_{j}$ would have a non-projective predecessor in 
$\Gamma(\mod A)$.
\end{proof}

\section{Stratifying systems over the algebra  $K(\widetilde{\mathbb{A}}_{p,q})$}

We denote by $(F, G, Y)$ an exceptional sequence where the elements of $F$ appear before the elements of $G$ 
and these last ones before $Y.$ We also use the same notation for the associated $s.s.$

\begin{prop}\label{posprojetivo1} If $Y$ is a postprojective $A$-module such that  $(F,G,Y)$ is a $s.s.$, then 
$Y$ is one of the modules in the following list:
\begin{enumerate}
\item $P_0$.
\item $P_{p+q-1}$.
\item $\tau^{-t}P_0$, with $t\geq 1$ such that  $p\mid t$ and $q\mid t$.
\item $\tau^{-t} P_{p+q-1}$, with  $t\geq 1$ such that $p\mid t$ and $q\mid t$.
\item $\tau^{-t}P_{p-r}$, with $t\geq 1$ such that $q\mid t$, $t\equiv r \,\,(\mod \,p)$ and $ r \in [1,p-1]$. 
\item $\tau^{-t}P_{p+q-r-1}$, with $t\geq 1$ such that $p\mid t$, $t\equiv r \,\,(\mod \,q)$ and $ r \in [1,q-1]$.
\end{enumerate}
\end{prop}

\begin{proof} If $M$ is a regular module, then
$$\Ext^{1}_{A}(\tau^{-t}P_j,M)\cong D\Hom_A(\tau^{-1}M,\tau^{-t}P_j)=0, \text{ for }t\geq 0.$$
Because of that, to prove that  $(F,G,\tau^{-t}P_j)$ is a $s.s.$ we need to show that 
$$\Hom_A(\tau^{-t}P_j,\FF)=0= \Hom_A(\tau^{-t}P_j, \GG).$$
On the other hand, as a consequence of Auslander-Reiten formula, we have that
$$\Hom_A(\tau^{-t}P_j, \FF)\cong \Hom_A(P_j, \tau^{t}\FF) \text{ and } \Hom_A(\tau^{-t}P_j,\GG) \cong \Hom_A(P_j,\tau^{t}\GG).$$
It follows that 
$(F,G,\tau^{-t}P_j)$ is a $s.s.$  if and only if $j \in (\Supp\tau^{t}\GG)' \cap (\Supp\tau^{t}\FF)',$  where 
$(\Supp\tau^{t}\GG)'$ denotes the complement of the set
$(\Supp\tau^{t}\GG)$ with relation to the set $[0,p+q-1]$.
\vspace{0.1cm}

\noindent {\bf Case 1.} If $t=0$, then $Y\cong P_j$ then $(F,G,P_j)$ is a $s.s.$ if and only if
$j \in (\Supp\,\FF)' \cap (\Supp\,\GG)'=\{0, p+q-1 \}$.

\noindent {\bf Case 2.} If  $p\mid t$ e $q \mid t$ then, according to Corollary \ref{suporte},
$\Supp\,\tau^{t}\FF= \Supp\,\FF$ and $\Supp\,\tau^{t}\GG= \Supp\,\GG.$
Therefore, $(F,G,\tau^{-t}P_j)$ is a $s.s.$ if and only if $j \in \{0,p+q-1 \}$.

\noindent {\bf Case 3.} If $q\mid t$ and  $t \equiv r\,\,(\mod\,p)$, with $r\in [1,p-1]$, 
then by Corollary \ref{suporte} we have that 
$(\Supp\,\tau^{t}\FF)' =\{p-r\} $  and 
$(\Supp\,\tau^{t}\GG)' =[0, p-1]\cup \{p+q-1\}.$
Therefore  $(F,G,\tau^{-t}P_j)$ is a $s.s.$ if and only if $j=p-r$.

\noindent{\bf Case 4.} If $p\mid t$ and  $q\nmid t$, then  similarly to the previous case we conclude that
$(F,G,\tau^{-t}P_j)$ is a $s.s.$ if and only if $j=p+q-r-1$, where 
 $t \equiv r \,\,(\mod \,q), r \in[1, q-1]$.

\noindent{\bf Case 5.} If $t \equiv r_1\,\,(\mod \,p)$ and  $t\equiv r_2\,\,(\mod\,q)$, with $r_1 \in [1,p-1] $ and 
$ r_2 \in[1,q-1]$  then $$(\Supp\, \tau^{r_1}\FF)' \cap (\Supp \tau^{r_2} \GG)'=\{p-r_1\} \cap \{p+q-r_2-1\}= \emptyset.$$ \end{proof}




\begin{prop}\label{preinjetivo}If $Y$ is a preinjective $A$-module such that $(F,G,Y)$ is a $s.s.$, then
$Y$ is one of the modules in the following list
\begin{enumerate}
\item $\tau^{t}I_{p}$, with $t\geq 1$ such that   $t\equiv p-1 (\mod p)$ and  $q\mid t$.
\item $\tau^{t}I_1$, with $t\geq 1$ such that  $t\equiv q-1(\mod q)$  and $p\mid t$. 
\item $\tau^{t}I_{0}$, with  $t\geq 1$ such that $t\equiv p-1 (\mod p)$  and  $t\equiv q-1 (\mod q)$.
\item $\tau^{t}I_{p+q-1}$, with  $t\geq 1$ such that $t\equiv p-1 (\mod p)$  and  $t\equiv q-1 (\mod q)$.
\item $\tau^{t}I_{r+1}$, with $t\geq 1$ such that $t\equiv r (\mod p)$, $r\in [1,p-2]$ and  $t\equiv q-1(\mod q)$.
\item $\tau^{t}I_{p+r}$, with $t\geq 1$ such that $t\equiv r (\mod q)$, $r \in [1,q-2]$  and  $t\equiv p-1(\mod p)$.
\end{enumerate}
\end{prop}
\begin{proof} Let $Y\cong \tau^{k}I_{j}$, for $k\geq 0$ and $j \in [0, p+q-1]$ and  $M$ be a regular $A$-module. Hence $\Hom_A(Y,M)=0$. On the other 
hand, from the  Auslander-Reiten formula,
we have $$\Ext^{1}_{A}(\tau^{k}I_j, M)\cong D\Hom_A(\tau^{-1}M,\tau^{k}I_j)\cong D\Hom_A(\tau^{-(k+1)}M,I_j).$$
It fo\-llows that
$\Ext^{1}_{A}(\tau^{k}I_j, M)=0  \text{ if and only if, } j \notin \supp\,\tau^{-(k+1)}M.$
Therefore
\[(F,G,\tau^{k}I_j) \text{ is a $s.s.$ } \text{ if and only if, } j \in (\Supp \tau^{-(k+1)}\FF)'\cap (\Supp \tau^{-(k+1)}\GG)'.\]
We consider several cases in order to find all the $s.s.$  of the  form  $(F,G,Y)$ with $Y$ preinjective.

\noindent {\bf Case 1.} Let $Y \cong I_j$ with  $j\in [0, p+q-1]$.
Then,  by Corollary \ref{suporte}, we  have that $(\Supp \tau^{-1}\FF)' \cap (\Supp \tau^{-1}\GG )' = \{1\}\cap\{p\}= \emptyset.$

\noindent {\bf Case 2.} Let $Y \cong \tau^{t}I_j$, with $t\geq 1$ such that  $p\mid t$ and $q\mid t$. Then
$$(\Supp\tau^{-(t+1)}\FF)' \cap  (\Supp\tau^{-(t+1)}\GG)'=(\Supp \tau^{-1}\FF)'\cap (\Supp \tau^{-1}\GG)' =\emptyset.$$

\noindent {\bf Case 3.} Let $Y \cong \tau^{t}I_j$, with $t$ such that $p \nmid t$ and  $q\mid t$. Let
$t \equiv r\,\,(\mod\, p)$, $ r \in [1, p-1]$. Here we need to consider two cases: 
\begin{itemize}
 \item If $r=p-1$ then, using  Corollary \ref{suporte}, we have that
$$ \Supp \tau^{-(t+1)}\FF =\Supp \tau^{-(r+1)}\FF = \Supp \tau^{-p}\FF = \Supp \FF  \text{ and }$$ 
$$\Supp \tau^{-(t+1)}\GG = \Supp \tau^{-1}\GG=[0,  p+q-1]\setminus\{p\}.$$
Hence $(\Supp \tau^{-(t+1)}\FF)' \cap (\Supp \tau^{-(t+1)}\GG)' =\{p\}$ and 
therefore  \\ $(F,G,\tau^t I_{p})$ is a $s.s.$

 \item If $r \not =p-1$, 
$\Supp\tau^{-(t+1)}\FF =\Supp\tau^{-(r+1)}\FF =[0,p+q-1]\setminus\{r+1\}.$
Then $(\Supp\tau^{-(t+1)}\FF)'\cap(\Supp\tau^{-(t+1)}\GG)' =\{r+1\}\cap \{p\}=\emptyset .$
\end{itemize}

\noindent {\bf Case 4.} Let $Y \cong \tau^{t}I_j$ with $t\geq 1$ such that $q \nmid t$  and $p\mid t$. Let
$r\equiv t\,\,(\mod q)$, with $r\in [1,q-1]$. Consider two cases $r=q-1$ and 
 $r\not= q-1$.
 \item\begin{itemize}
 \item If $r=q-1$, then $\Supp \tau^{-(t+1)}\GG  =\Supp \tau^{-q}\GG= \Supp \GG \text{ and }$
$\Supp \tau^{-(t+1)}\FF = \Supp \tau^{-1}\FF =[ 0, p+q-1]\setminus \{1\}.$
It follows that $(\Supp  \tau^{-(t+1)}\FF)' \cap (\Supp \tau^{-(t+1)}\GG)' =\{1\}.$ Therefore $(F,G,\tau^t I_{1})$
is a $s.s.$
\item If $r \not =q-1$, so we have that
$$\Supp \tau^{-(t+1)}\GG = \Supp \tau^{-(r+1)}\GG  = [0, p+q-1] \setminus \{p+r\}. $$
Then $(\Supp \tau^{-(t+1)}\FF)'\cap (\Supp \tau^{-(t+1)}\GG)' = \{1\}\cap \{p+r\}=\emptyset$.
\end{itemize}

\noindent{\bf Case 5.} Let $Y \cong \tau^{t}I_j$, with $t$ such that $q \nmid t$ and $p \nmid t$.
Let $r_1$ and  $r_2$  be such that $t \equiv r_1 \,\,(\mod\,p)$, $t \equiv r_2 \,\,(\mod\,q)$,
$ r_1 \in [1,p-1]$ and $r_2 \in [1,q-1]$. We  will consider several cases:
\begin{itemize}
\item If $r_1=p-1$ and $r_2 = q-1$, then we have:
$$(\Supp \tau^{-(t+1)}\FF)'\cap (\Supp \tau^{-(t+1)}\GG)' =(\Supp \FF)' \cap (\Supp \GG)' =\{ 0, p+q-1\}.$$
Therefore $(F, G, \tau^{t}I_0)$ and $(F, G,\tau^{t}I_{p+q-1})$ are $s.s.$
\item If $r_2=q-1$ and  $r_1 \not = p-1$. Then
$\Supp \tau^{-(t+1)}\GG = \Supp \GG  \text{ and  }$                
$$\Supp \tau^{-(t+1)}\FF = \Supp \tau^{-(r_{1}+1)}\FF =[0, p+q-1] \setminus \{r_{1}+1\}. $$
Hence $(F, G,\tau^t I_{r_{1}+1})$ is a $s.s.$
\item If $r_1=p-1$ and $r_2 \not = q-1$, then 
$\Supp \tau^{-(t+1)}\FF =\Supp \FF \text{ and  }$
$$\Supp \tau^{-(t+1)}\GG = \Supp \tau^{-(r_{2} +1)}\GG  =[0,p+q-1] \setminus \{p+r_{2}\}.$$
Then $(F, G,\tau^t I_{p+r_{2}})$ is a $s.s.$
\item If $r_1 \not=p-1$ and $r_2 \not= q-1$, then
$$\Supp \tau^{-(t+1)}\GG = \Supp \tau^{-(r_{2} +1)}\GG =[0, p+q-1] \setminus \{p+r_{2}\} \text{ and }$$
$$\Supp \tau^{-(t+1)}\FF = \Supp \tau^{-(r_{1} +1)}\FF =[0, p+q-1]\setminus \{r_{1}+1\}. $$
Consequently, there is no  $s.s.$ in this case. \end{itemize}
\end{proof}




The following result is well-known and it is not difficult to prove.

\begin{lem}\label{moe} Given a finite dimensional  $K$-algebra $B$ and
\begin{displaymath}
\xymatrix{
0 \ar[r] & L \ar[r]^{(f' \,\, g')^{t}} & X \oplus Y \ar[r]^{(f\,\,g)}&  M \ar[r]& 0 
}
\end{displaymath}
an almost split sequence in $\mod B$ . Then
\begin{enumerate}
\item If $f'$ is a monomorphism (resp. epimorphism), then $g$ is a monomorphism (resp. epimorphism).
\item If $g'$ is a monomorphism (resp. epimorphism), then $f$ is a monomorphism (resp. epimorphism).
\end{enumerate}
\end{lem}
\begin{prop}\label{fcpp} Let $A=K(\widetilde{\mathbb{A}}_{p,q})$. The postprojective component $\PP(A)$ of $\Gamma(\mod A)$ has 
the following properties:
\begin{enumerate}
\item  Any  irreducible morphism $W \lfd V$ between indecomposable modules in  $\PP(A)$ is a monomorphism.
\item The module $P_{p+q-1}$ and  its successors in $\PP(A)$ are sincere.
\item If $0<r<p$ then there is a projective $P_i$, for some $i\in [0, p+q+1]$ such that $\tau^{-r}P_i$ is not sincere, moreover $\tau^{-t}P_i$
is sincere for all $t\geq p$ and $i\in [0, p+q+1]$.
\item If $i \in[0,p-1]$, then
\begin{itemize}
      \item The smallest integer $r$ such that $\tau^{-r}P_i$ is sincere is $r=p-i$. Furthermore, all
       the modules of the form $\tau^{-k}P_i$, with  $k>p-i$, are sincere.
      \item If $ k \in [0,p-i]$, then the composition factors of  $\tau^{-k}P_i$ are
      $S_j$ with $ j \in [0,i+k] \cup [p,p+k-1]$.
      \item The composition factors  of $P_i$ are $S_j$ with  $j \in [0,i]$.
      \end{itemize}
\item If $ i \in [p,p+q-2]$ and $p+q > 2$ then
      \begin{itemize}
      \item If $i\leq q-1$, then the smallest integer $r$ such that $\tau^{-r}P_i$ is sincere is
       $r=p.$ Furthermore, all the modules $\tau^{-k}P_{i}$, with $k>p$, are sincere.
      \item If $i \in [q-1, p+q-1]$, then the smallest integer $r$ such that $\tau^{-r}P_i$ is sincere 
      is $r=p+q-1-i.$ Furthermore, all the modules $\tau^{-k}P_{i}$, with $k>p+q-1-i$, are sincere.
      \item If $\tau^{-k}P_{i}$ is not sincere, then its composition factors are $S_j$ with $ j \in [p, i+k] \cup [0, k]$.
      \end{itemize} 
     \item If $ k \in [1,p-1]$, then the composition factors of $\tau^{-k}P_{0}$ are $S_j$ with  $j \in [0,k]\cup[p, p-1+k]$.

\end{enumerate}

\end{prop}
\begin{proof} Our proof uses the structure of the postprojective component $\PP(A)$ of $\Gamma(\mod A)$ which looks as follows:
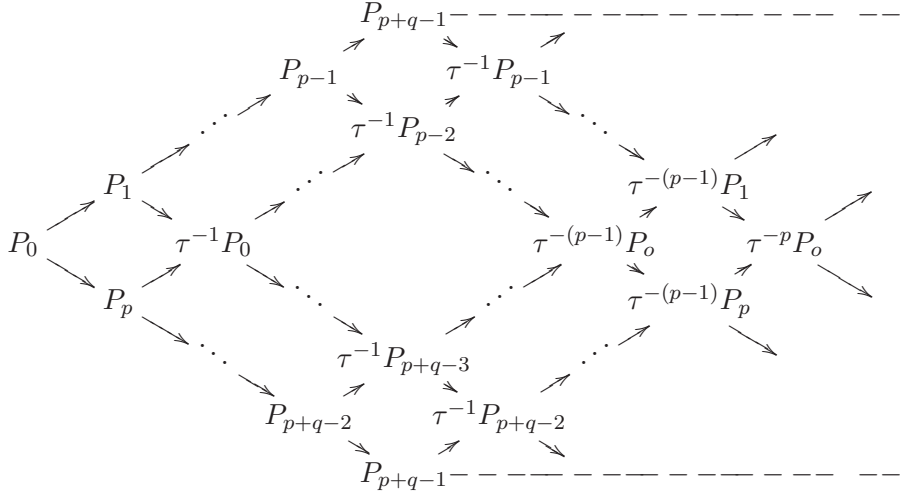
\begin{displaymath}
\xymatrix @!0 @R=2em @C=3pc{
& & & & P_{p+q-1} \ar[dr]&----&---- &----&----&--\\
& & & P_{p-1}\ar[dr]\ar[ur] & & \tau^{-1}P_{p-1}\ar[dr]\ar[ur] & & &&\\
& & \iddots \ar[ur] & & \tau^{-1}P_{p-2}\ar[ur]\ar[dr] & & \ddots \ar[dr] & &&\\
& P_{1} \ar[ur]\ar[dr] & & \iddots  \ar[ur] && \ddots \ar[dr] & & \tau^{-(p-1)}P_{1}\ar[dr] \ar[ur] &&\\
P_0 \ar[ur]\ar[dr] & &\tau^{-1}P_{0}\ar[ur] \ar[dr]& & & & \tau^{-(p-1)}P_{o}\ar[dr]\ar[ur] & & \tau^{-p}P_{o} \ar[ur] \ar[dr]&\\
& P_{p}\ar[ur]\ar[dr] & & \ddots \ar[dr] & & \iddots \ar[ur]&& \tau^{-(p-1)}P_p \ar[ur]\ar[dr]&&\\
& &\ddots\ar[dr] & & \tau^{-1}P_{p+q-3}\ar[ur] \ar[dr] & & \iddots \ar[ur] &&&\\
& & &P_{p+q-2}\ar[ur]\ar[dr]& & \tau^{-1}P_{p+q-2}\ar[ur]\ar[dr]&&&&\\
& & & &P_{p+q-1}\ar[ur]&----&----&----&----&--\\
}
\end{displaymath}

In order to prove (1) we observe that in the almost split sequence
$$ 0 \lfd P_0 \overset{(\alpha'\,\,\,\beta')^{t}} \lfd P_1\oplus P_p \overset{(\alpha\,\,\,\beta)}\lfd \tau^{-1}P_0 \lfd 0,$$
the morphisms $\alpha'$ and $\beta'$ are mono, because are irreducible morphisms between projective modules. Then, by  Lemma \ref{moe}, 
$\alpha$ and  $\beta$ are mono too. Analogously, in all almost split sequences beginning in a projective module the arrows represent 
monomorphisms. Hence, according to the shape $\PP(A)$, all the arrows in this component represent monomorphisms.

Now we prove (2). We note that in $\PP(A)$ there are paths 
$$P_0\fd P_1\fd\cdots\fd P_{p-1}\fd P_{p+q-1} \text{ and  }P_0 \lfd P_p \fd \cdots\fd P_{p+q-1} \fd P_{p+q-1},$$
where, by (1), all the arrows represent monomorphisms. Therefore, $P_{p+q-1}$ and  its successors are  sincere modules. 

For (3), we observe that none of the predecessors  in  $\PP(A)$ of modules in the paths  below 
\begin{equation}\label{caminho1}
P_{p+q-1}\lfd \tau^{-1}P_{p-1} \lfd \tau^{-2}P_{p-2}\lfd \cdots \lfd \tau^{-(p-1)}P_{0} \lfd  \tau^{-p}P_{0} \end{equation}
\begin{equation}\label{caminho2}
P_{p+q-1}\lfd \tau^{-1}P_{p+q-2} \lfd \tau^{-2}P_{p+q-3}\lfd \cdots \lfd \tau^{-(p-1)}P_{p} \lfd \tau^{-p}P_{0} \end{equation}
have  $S_{p+q-1}$ as a composition factor. In particular, $\tau^{-(p-1)}P_{0}$ does not have $S_{p+q-1}$ as a composition factor. 
Therefore, by (1), none of its predecessors have $S_{p+q-1}$ as a composition factor. On the other hand, any successor of 
$\tau^{-(p-1)}P_{0}$, and any of modules in (\ref{caminho1}) and (\ref{caminho2}), are sincere.

In order to show (4), we note that all the modules  $\tau^{-(p-i)}P_{i}$, with $ i \in [0,p-1]$,
are in the path (\ref{caminho1}) and therefore are sincere. Moreover, as stated earlier, no predecessor of a module which is in the path
(\ref{caminho1}) has  $S_{p+q-1}$ as a composition factor. It follows that, if 
$ i \in [0,p-1]$ then none of the modules of the form $\tau^{-k}P_{i}$, with $k<p-i$, are sincere. 

On the other hand, the predecessors of  $P_i$, with  $ i \in [0,p-1]$, are the projectives
$P_0, P_1,\ldots,P_{i-1}$. Thus the composition factors of $P_i$,  with $ i \in [0,p-1]$,
are $S_j$ with $ j \in [0,i]$. Analogously, for $ i \in [p,i <p+q-2]$, the composition factors of
$P_i$ are $S_j$, with $ j \in [p,i] \cup \{0\}$. 

Let $i \in[0, p-1] $. Then in  $\PP(A)$ there are paths of the form
\begin{equation}
P_{i+k}\fd \tau^{-1}P_{i+k-1} \fd \tau^{-2}P_{i+k-2}\fd \cdots \fd \tau^{-(k-1)}P_{i+1} \fd  \tau^{-k}P_{i} \end{equation} 
and 
\begin{equation}
\begin{array}{l}
 P_{p+k-1}\fd \tau^{-1}P_{p+k-2} \fd \cdots \fd \tau^{-(k-1)}P_{p} \fd \tau^{-k}P_{0} \fd 
  \tau^{-k}P_{1}\fd            \cdots \fd \tau^{-k}P_{i}
\end{array}
\end{equation}

Furthermore,  there are no paths between $P_{j}$, with $j \geq i+k$, and  $\tau^{-k}P_{i}$. Hence the composition factors of 
$\tau^{-k}P_i$, up to multiplicity, are the composition factors of $P_{i+k}$ and of $P_{p+k-1}.$ In other words, the composition 
factors of  $\tau^{-k}P_{i}$ are  $S_j$ with $j \in [0,i]$ and $ i \in [p,p+q-2]$.

The proof of (5), is analogous to the previous  and part (6) is left to the reader, since it follows by a similar argument.
\end{proof}

\begin{prop}\label{completo1} Let $A=K(\widetilde{\mathbb{A}}_{p,q})$. The  complete list of $s.s.$
$(M,F,G,Y)$, where $Y$ is a postprojective $A$-module, is as follows:
\begin{enumerate}
\item $(S_{p+q-1},F,G, P_0)$.
\item $(M,F,G, P_{p+q-1})$, where
      \begin{displaymath}
      \begin {array}{rll}
      M \cong &\left \{ \begin {array}{ll}
      \tau^{-p+1}P_{0}, \textrm{ if } p=q\\
      \tau^{-p+1}P_{q-1}, \textrm{ if } p \not =q\\
      \end{array} \right.\\
      \end{array}
      \end{displaymath}
\item $(M,F,G,\tau^{-t} P_{p+q-1})$ with $t\in \mathbb{N}$ such that $p\mid t$ and $q\mid t$, where
\begin{displaymath}
      \begin {array}{rll}
      M \cong &\left \{ \begin {array}{ll}
      \tau^{-t-p+1}P_{0}, \textrm{ if } p=q\\
      \tau^{-t-p+1}P_{q-1}, \textrm{ if } p \not =q\\
      \end{array} \right.\\
      \end{array}
\end{displaymath}
\item $(\tau^{-t+1}P_{p+q-1},F,G,\tau^{-t}P_0)$, with $t\in\mathbb{N}$
      such that $p\mid t$ and $q \mid t$.
\item $(\tau^{-t-(p-r-1)}P_{q+r-1},F,G,\tau^{-t} P_{p-r})$, with $t,r\in \mathbb{N}$ such that $q\mid t$, $t\equiv r \,\,(\mod \,p)$ 
      and $ r \in [1,p-1]$. 
\item $(M, F,G,\tau^{-t}P_{p+q-r-1})$, with $t,r\in\mathbb{N}$ such that $p\mid t$, $t\equiv r \,\,(\mod \,q)$
      and $r \in [1,q-1]$, where
      \begin{displaymath}
      \begin {array}{rll}
      M \cong &\left \{ \begin {array}{ll}
      \tau^{-t-q+r+1}P_{p-q+r},  \textrm{ if } p \geq q-r\\
       \tau^{-t-p+1}P_{q-r-1}, \textrm{ if } p<q-r\\
      \end{array} \right.\\
      \end{array}
\end{displaymath}
\end{enumerate}
\end{prop}
\begin{proof}  Proposition \ref{posprojetivo1} characterizes the $s.s.$ of the form $(F, G, Y)$,  with $Y$
a  postprojective $A$-module. Using this characterization we obtain stratifying systems of type $(X,F,G,Y)$.
 
The quiver $\widetilde{\mathbb {A}}_{p,q}$ has $p+q$ vertices then, by Lemma \ref{estender} and  Lemma \ref{unico}, there is a unique 
 module $M$ such that  $(M,F,G,\tau^{-t}P_l)$ is a $c.s.s.$  Thus, the proof consists in the verification
that the indecomposable module $ M $ satisfies the following conditions:
\begin{itemize}
\item $\Ext_{A}^{1}(M,M)=0$.
\item  $\Hom_{A}(\FF,M)=0$ and $\Hom_{A}(\GG,M)=0$.
\item  $\Ext_{A}^{1}(\FF,M)=0$ and $\Ext_{A}^{1}(\GG,M)=0$.
\item  $\Hom_{A}(\tau^{-t}P_l,M)=0$ and $\Ext_{A}^{1}(\tau^{-t}P_l,M)=0$.
\end{itemize}
We observe that all the modules $M$ considered in this proof, except $I_ {p+q-1}$,
are of the form $M\cong\tau^{-k}P_j$, with $k \geq 0$, and therefore are indecomposable and have no self-extensions.
Moreover, there is no nonzero morphism from a regular module to a posprojective module, that is, if $R$ is a regular 
module then $\Hom_A(R,\tau^{-k}P_j)=0$. On the other hand, by the Auslander-Reiten formula, we have that 
$$\Ext^{1}_A(R, \tau^{-k}P_{j}) \cong  D\Hom_A(\tau^{-k}P_{j}, \tau R)\cong D\Hom_A(P_{j}, \tau^{k+1}R).$$
Therefore we have that $\Ext^{1}_A(\FF,\tau^{-k}P_{j})=0 \text{ and }\Ext^{1}_A(\GG,\tau^{-k}P_{j})=0$ if and only if, 
$ j\in (\Supp \tau^{k+1}\FF)' \cap (\Supp \tau^{k+1} \GG)'$.
In view of these observations, if $(F,G,\tau^{-t}P_l)$ is a $s.s.$ and
$X\cong\tau^{-k}P_j$, in order to show that $(X, F,G,\tau^{-t}P_l)$
is a $s.s.$ it is sufficient to check the following conditions:                                                                            
\begin{itemize}
\item $j\in (\Supp \tau^{k+1}\FF)' \cap (\Supp \tau^{k+1} \GG)'.$
\item $\Hom_{A}(\tau^{-t}P_l,M)=0$  (or equivalently $[\tau^{t}M:S_l]=0$).
\item  $\Ext_{A}^{1}(\tau^{-t}P_l,M)=0$.
\end{itemize}
In what follows we find these conditions for each of the sequences stated in the proposition.

\noindent (1) First let us check that $(S_{p+q-1},F,G, P_0)$  is a $s.s.$ As the vertex $p+q-1$ is a source then 
the $S_{p+q-1} \cong I_{p+q-1}$ and thus 
      $$\Ext_{A}^{1}(\FF,S_{p+q-1})=0,\Ext_{A}^{1}(\GG,S_{p+q-1})=0 \text{ and }\Ext_{A}^{1}(P_{0},S_{p+q-1})=0.$$
On the other hand,  since  $p+q-1\notin(\Supp\,\FF \cup\Supp\,\GG)$   then \\
    $\Hom_{A}(\FF,S_{p+q-1})=0,$ $ \Hom_{A}(\GG, S_{p+q-1})=0$  and  $\Hom_{A}(P_0, S_{p+q-1})=0.$
    
\noindent (2) To complete the $s.s.$ $(F,G, P_{p+q-1})$  we have to consider two cases:
      \begin{itemize}
      \item If $p=q$ and $M \cong \tau^{-p+1}P_0$, then  by Corollary \ref{suporte} we have that
      $(\Supp \tau^{p}\FF)' \cap (\Supp \tau^{p} \GG)'= \{0,p+q-1\}.$
      On the other hand $[\tau^{-p+1}P_{0}:S_{p+q-1}]=0$, by Proposition \ref{fcpp} (6).
       Therefore  \\$(\tau^{-p+1}P_0,F,G,P_{p+q-1})$
      is a $c.s.s.$
    
      \item If $p\not=q$ and $M\cong \tau^{-p+1}P_{q-1}$ then $$(\Supp\tau^{p}\FF)' \cap (\Supp\tau^{p}\GG)'= \{q-1\}$$
            and, by Proposition \ref{fcpp} (5), we have that $[\tau^{-p+1}P_{q-1}:S_{p+q-1}]=0$.
            \end{itemize}
 \vspace{0.3cm}

\noindent (3) In order to complete the $s.s.$ $(F,G,\tau^{-t} P_{p+q-1})$, with $t\geq 1$ such that
      $p\mid t$ and  $q\mid t$, we consider the following two cases:
\begin{itemize}
\item Suppose $p=q$. We claim that $(\tau^{-t-p+1}P_{0},F, G,\tau^{-t}P_{p+q-1})$, is a $c.s.s.$ In fact,
      $ (\Supp\tau^{t+p}\FF)' \cap (\Supp\tau^{t+p}\GG)'=\{0, p+q-1\}.$ On the other hand, by Proposition \ref{fcpp} (6), 
      $\Hom_A(\tau^{-t}P_{p+q-1},\tau^{-t-p+1}P_{0})\cong 0.$
      Moreover, by  Lemma \ref{repete}, $\Ext^{1}_A(\tau^{-t}P_{0},\tau^{-t-p+1}P_{q-1})=0.$
    
      \item Suppose $p\not =q$. In this case $(\tau^{-t-p+1}P_{q-1},F,G,\tau^{-t}P_{p+q-1})$
      is a $c.s.s.$ Indeed, $(\Supp\tau^{t+p}\FF)'\cap (\Supp\tau^{t+p}\GG)'=\{q-1\}$, by Corollary \ref{suporte}, 
      and $[\tau^{-p+1}P_{q-1}:S_{p+q-1}]=0$, by Proposition  \ref{fcpp} (5). Finally, by Lemma \ref{repete}, we have that
      $\Ext^{1}_A(\tau^{-t}P_{p+q-1},\tau^{-t-p+1}P_{q-1})=0$.
      \end{itemize} 
       \vspace{0.3cm}
      
\noindent (4) We show that $(\tau^{-t+1}P_{p+q-1},F, G, \tau^{-t}P_{0})$, with $t\geq 1$ such that
      $p \mid t$ and  $q\mid t$ is a $c.s.s.$ In fact,
      $(\Supp \tau^{t}\FF)' \cap (\Supp \tau^{t} \GG)'= \{ 0, p+q-1\}.$
       And by the Auslander-Reiten formula we have 
      $$\Hom_A(\tau^{-t}P_{0},\tau^{-t+1}P_{p+q-1})\cong D\Ext^{1}_{A}(P_{p+q-1},\tau^{-1}P_{0})\cong 0,$$
     $$ \Ext_A^{1}(\tau^{-t}P_{0},\tau^{-t+1}P_{p+q-1}) \cong D\Hom_A(P_{p+q-1},P_{0}) \cong 0. $$
     Then the conditions for the $s.s.$ are verified.
      \vspace{0.3cm}
      
\noindent (5) We prove that $(\tau^{-t-(p-r-1)}P_{q+r-1},F,G,\tau^{-t} P_{p-r})$, with $t,r\in \mathbb{N}$ such that  $q \mid t$,          
      $t\equiv r \,\,(\mod \,p)$ and $ r \in[1,p-1]$ is a $s.s.$ According to  Corollary \ref{suporte}, we have 
      $$(\Supp \tau^{t+(p-r)}\FF)' \cap (\Supp \tau^{t+(p-r)} \GG)'=(\Supp \FF)' \cap (\Supp \tau^{p-r} \GG)'
                                                                 =\{q+r-1\}.$$                     
      Applying  Lemma \ref{repete} we can  assert that 
      $$\Ext_A^{1}(\tau^{-t}P_{p-r},\tau^{-t-(p-r-1)}P_{q+r-1})=0.$$ And  finally, by Proposition \ref{fcpp} (5), we have 
    $$\Hom_A(\tau^{-t}P_{p-r},\tau^{-t-(p-r-1)}P_{q+r-1})\cong \Hom_A(P_{p-r},\tau^{-(p-r-1)}P_{q+r-1})\cong 0.$$

\noindent (6) Suppose that $t,r\in\mathbb{N}$ are such that $p\mid t$, $t\equiv r \,\,(\mod \,q)$
      and $ r \in [1, q-1]$. Let us consider two cases:
     \begin{itemize}
     \item Assume that $p \geq q-r$. Then, by Corollary \ref{suporte}, we have that
    \[ \begin{array}{rcl}
      (\Supp\tau^{t+q-r}\FF)' \cap (\Supp \tau^{t+q-r}\GG)'& =& (\Supp \tau^{q-r} \FF)'\cap(\Supp\GG)'\\
                                                             & = &\{p-q+r\}.\\
      \end{array}\]
      We note that  $q-r-1 \geq 0 $, then  by  Lemma \ref{repete}  we  have that
      $$\Ext_A^{1}(\tau^{-t}P_{p+q-r-1},\tau^{-t-q+r+1}P_{p-q+r})=0.$$
     Furthermore, by  Proposition \ref{fcpp} (4), it follows that
      \[\begin{array}{rcl}
      \Hom_A(\tau^{-t} P_{p+q-r-1},\tau^{-t-q+r+1} P_{p-q-r})&\cong&\Hom_A(P_{p+q-r-1},\tau^{-q+r+1}P_{p-q+r})
      \\ &\cong& 0.       
      \end{array}\]
   Thus $(\tau^{-t-q+r+1} P_{p-r_1}, F,G, \tau^{-t} P_{p+q-r-1})$ is a $c.s.s.$
\item Suppose that $p<q-r$. Let $l$ such that  $p+l=q-r$. Therefore, by Corollary \ref{suporte}, we have that
     \[\begin{array}{rcl}
     (\Supp \tau^{t+p}\FF)' \cap (\Supp \tau^{t+p} \GG)'&=& (\Supp \FF)' \cap (\Supp  \tau^{-l}\GG)' \\
                                                            &=&\{p+l-1\}\\
                                                            &=&\{q-r-1\}.\\\end{array}\]
     Now, by Lemma \ref{repete}, $\Ext_A^{1}(\tau^{-t}P_{p+q-r-1},\tau^{-t-p+1}P_{q-r-1})=0.$
     Finally, by  Proposition \ref{fcpp} (4),  we have that $S_{p+q-r-1}$  is not a composition factor of  $\tau^{-p+1}P_{q-r-1}.$
      \end{itemize}
\end{proof}

Next result is proved in an analogous way as Proposition \ref{fcpp}.  

\begin{prop}\label{preinjecte}Let $A=K(\widetilde{\mathbb{A}}_{p,q})$. The preinjective component $\QQ(A)$ of $\Gamma(\mod A)$ has the following 
properties:
\begin{enumerate}
\item Any irreducible  morphism $W\lfd V$ between indecomposable modules in $\QQ(A)$ is an epimorphism.
\item The $A$-module $I_{0}$  and  its predecessors in $\QQ(A)$ are  sincere.
\item If $0<r<p$ then there is an injective $I_i$, for some $i\in [0, p+q+1]$ such that $\tau^{r}I_i$ is not sincere, moreover $\tau^{t}I_i$
is sincere for all $t\geq p$ and $i\in [0, p+q+1]$.
\item If $ i \in [1,p-1]$, then:
      \begin{itemize}
      \item the smallest integer $r$ such that $\tau^{r}I_{i}$ is a sincere module is $r=i$. 
      Furthermore, all the modules of the form $\tau^{k}I_{i}$, with $k\geq i$, are  sincere. 
      \item if $k<i$, then the composition factors of $\tau^{k}I_{i}$ are  simple   $S_j$ with $  j \in [i-k,p-1] \cup [p+q-k, p+q-1]$.      
      \end{itemize}  
\item If $p\leq i< p+q-1$, then:
      \begin{itemize} 
      \item if  $p\leq i< q-1$, the smallest integer $r$ such that $\tau^{r}I_{i}$ is a sincere module is 
      $r=i-p+1.$  Moreover,  all the modules $\tau^{r}I_{i}$, with $r \geq i-p+1$,
      are sincere.           
      \item if $ i \in[q-1, p+q-1]$, the smallest integer $r$ such that $\tau^{r}I_{i}$ is a sincere $A$-module is  $r=p$. 
      Furthermore, all the modules $\tau^{r}I_{i}$, with $r\geq p$ are sincere.
      \item if $ i \in [p, p+q-2]$ and  $k$ is such that  $\tau^{k}I_{i}$ is not a sincere module, then the composition
      factors of  $\tau^{k}I_{i}$ are simple modules $S_j$ with $ j \in [i-k,p+q-1] \cup[p-k, p-1]$.
     
      \end{itemize}
\item If $k<p$, then  the composition of $\tau^{k}I_{p+q-1}$ are simple modules  $S_j$ with $ j \in[p-k,p-1] \cup [p+q-k, p+q-1]$.

\end{enumerate}
\end{prop}
\begin{prop}\label{completo2} \label{fcpi} Let $A=K(\widetilde{\mathbb{A}}_{p,q})$. The complete list of $c.s.s.$
$(M,F,G,Y)$, where $Y$ is a preinjective $A$-module is the following:
\begin{enumerate}

\item $(\tau^{t}I_{p-1},F,G,\tau^{t}I_{p})$, with $t\geq 1$ such that $t\equiv p-1\,\,(\mod\,p)$ and $q|t$.
\item $(\tau^{t}I_{p+q-2},F,G,\tau^{t}I_1)$, with $t\geq 1$ such that $t\equiv q-1\,\,(\mod\,q)$ and $p|t$. 
\item $(\tau^{t+1}I_{p+q-1},F,G,\tau^{t}I_{0})$, with $t\geq 1$ such that $t\equiv p-1\,\,(\mod\,p)$ and  $t\equiv q-1\,\,(\mod\,q)$.
\item $(\tau^{t-p+1}I_{q-1},F,G,\tau^{t}I_{p+q-1})$, with $t\geq 1$ such that $t\equiv p-1\,\,(\mod\,p)$ and $t\equiv q-1 \,\,(\mod \,q)$.
\item $(\tau^{t-r} I_{p+q-r-2},F,G,\tau^{t}I_{r+1})$, with $t\geq 1$ such that $t\equiv r\,\,(\mod\,p)$,
      $r \in[1,p-2]$ and $t\equiv q-1\,\,(\mod\,q)$.
\item $(M,F,G,\tau^{t}I_{p+r})$, with $t\geq 1$ such that $t\equiv r\,\,(\mod\,q)$, $  r \in [1,q-2]$ and $t\equiv p-1\,\,(\mod\,p)$, where 
      \begin{displaymath}
      \begin {array}{rll}
      M \cong &\left \{ \begin {array}{ll}
      \tau^{t-r}I_{p-(r+1)}, \textrm{ if } r < p\\
      \tau^{t-(p-1)}I_{r}, \textrm{ if } r \geq p\\
      \end{array} \right.\\
      \end{array}
      \end{displaymath}

\end{enumerate}
\end{prop}

\begin{proof} By Lemma \ref{estender} and Lemma \ref{unico} there is a unique module $M$ such that $(M,F,G,\tau^{m}I_i)$ 
is a $s.s.$ Therefore the proof consists in  verify
that the indecomposable module $ M $ satisfies the following conditions: 
\begin{itemize}
\item $\Ext_{A}^{1}(M,M)=0$.
\item  $\Hom_{A}(\FF,M)=0  \text{ and }\Hom_{A}(\GG,M)=0$.
\item  $\Ext_{A}^{1}(\FF,X)=0 \text{ and  } \Ext_{A}^{1}(\GG,M)=0$.
\item  $\Hom_{A}(\tau^{t}I_i,M)=0 \text{ and }\Ext_{A}^{1}(\tau^{t}I_i,M)=0$.
\end{itemize}

All modules $M$ that we will consider are of the form $M\cong\tau^{l}I_j$ and therefore are
indecomposable and have no self-extensions. Moreover if  $M$ is a preinjective and $R$ is a regular then, using the Auslander-Reiten formula, 
we have $\Ext^{1}_A(R,M)\cong D\Hom_A(M,\tau R)\cong 0$.
On the other hand, for $l\geq 0$, we have $\Hom_A(R,\tau^{l}I_j) \cong \Hom_A(\tau^{-l}R,I_j).$ It follows that
$\Hom_A(R,\tau^{l}I_j)=0$ if and only if $j \notin \supp\tau^{-l}R.$ Therefore we have that
$\Hom_{A}(\FF,\tau^{l}I_j)=0$ \text{ and }  $\Hom_{A}(\GG,\tau^{l}I_j)= 0$ \text{ if and only if } $j \in (\Supp\tau^{-l}\FF)'\cap 
(\Supp\tau^{-l}\GG)'.$

According to the above remarks, if $(F,G,\tau^{t}I_i)$ is a $s.s.$ and $M\cong\tau^{l}I_j$, in order to prove 
that $(M,F,G,\tau^{t}I_i)$ is a $c.s.s.$ it is sufficient to show the following conditions: 
$j \in (\Supp \tau^{-l}\FF)'\cap (\Supp \tau^{-l}\GG)'$, $\Hom_{A}(\tau^{t}I_i,M)=0$ and $\Ext_{A}^{1}(\tau^{t}I_i,M)=0$.

We consider various possibilities, which depend on the form of $t$, according to the list of the statement of Proposition  \ref{preinjetivo}.

\noindent(1) Let $t\geq 1$ such that $t\equiv p-1\,\,(\mod \,p)$ and $q\mid t$. Then, by Corollary \ref{suporte}, we have 
      $$(\Supp \tau^{-t}\FF)' \cap(\Supp \tau^{-t}\GG)'=(\Supp \tau^{-(p-1)}\FF)'\cap(\Supp \tau^{-t}\GG)'= \{p-1\}.$$
      We observe that $\Hom_A(\tau^{t}I_{p},\tau^{t}I_{p-1})\cong \Hom_A(I_{p},I_{p-1}) \cong 0$, because $I_p$ has no $S_{p-1}$
      as a composition factor.
         
      Moreover, by Lemma \ref{repete}, $\Ext^{1}_A(\tau^{t}I_{p},\tau^{t}I_{p-1})=0.$  Then 
      $(\tau^{t}I_{p-1},F,G,\tau^{t}I_{p})$ is a $s.s.$

\noindent(2) Let $t\geq 1$ such that $t\equiv q-1\,\,(\mod \,q)$ and $p\mid t$. By Corollary \ref{suporte} we have that
      $$(\Supp \tau^{-t}\FF)'\cap(\Supp \tau^{-t}\GG)'=(\Supp \tau^{-t}\FF)'\cap(\Supp \tau^{-(q-1)}\GG)'=\{p+q-2\}.$$
      On the other hand, since $S_{p+q-2}$ is not a composition factor of $I_1$ we see that 
$\Hom_A(\tau^{t}I_{1},\tau^{t}I_{p+q-2})\cong \Hom_A(I_{1},I_{p+q-2})\cong 0$.
     
   Finally, by Lemma \ref{repete}, it follows that $\Ext^{1}_A(\tau^{t}I_{1},\tau^{t}I_{p+q-2})=0.$
      
\noindent(3) Let $t\geq 1$ such that $t\equiv p-1\,\,(\mod \,p)$ and $t\equiv q-1\,\,(\mod\,q)$. By  Corollary \ref{suporte} we have that
     $$(\Supp \tau^{-t-1}\FF)'\cap(\Supp \tau^{-t-1}\GG)'= (\Supp \FF)'\cap(\Supp \GG)'=\{0, p+q-1\}.$$
    
     From the Auslander-Reiten formula it may be concluded that 
     $$\Hom_A(\tau^{t}I_{0},\tau^{t+1}I_{p+q-1})\cong\Hom_A(I_{0},\tau I_{p+q-1})\cong D\Ext_A^{1}(I_{p+q-1},I_{0})=0 \text{ and }$$
    \[\begin{array}{rcl}
    \Ext^{1}_A(\tau^{t}I_{0},\tau^{t+1}I_{p+q-1})&\cong& D\Hom_{A}(\tau^{t+1}I_{p+q-1},\tau^{t+1}I_{0})\\
                                                 &\cong& \Hom_A(I_{p+q-1},I_{0})\\
                                                  &=&0.
     \end{array}\]
     
\noindent(4) Let  $t\geq 1$ such that  $t\equiv p-1 \,\,(\mod \,p)$ and  $t\equiv q-1 \,\,(\mod \,q)$. We show that 
      $(\tau^{t-p+1}I_{q-1},F,G,\tau^{t}I_{p+q-1})$ is a $c.s.s.$ over $A$. By
      Corollary \ref{suporte} it follows that
      $$ (\Supp \tau^{-t+p-1}\FF)'\cap(\Supp \tau^{-t+p-1}\GG)'=(\Supp \FF)'\cap(\Supp \tau^{-(q-p)}\GG)'
                                                           =\{q-1\}.$$
      Since, $\Hom_A(\tau^{t}I_{p+q-1},\tau^{t-(p-1)}I_{q-1})\cong \Hom_A(\tau^{p-1}I_{p+q-1},I_{q-1})=0,$
      because of Proposition \ref{preinjecte} (5)  we get  that  $[\tau^{p-1}I_{p+q-1}:S_{q-1}]=0.$  Moreover, by  Lemma \ref{repete}, 
      we have that  $\Ext^{1}_A(\tau^{t}I_{p+q-1},\tau^{t-(p-1)}I_{q-1})=0.$
      
\noindent(5) Let $t\geq 1$ such that $t\equiv r\,\,(\mod\,p)$, $ r\in[1,p-2]$ and $t\equiv q-1\,\,(\mod\,q)$ and 
      $M\cong \tau^{t-r} I_{p+q-r-2}$.  First, by Corollary \ref{suporte}, we have that
      \[\begin{array}{rcl}
      (\Supp\tau^{-(t-r)}\FF)'\cap(\Supp\tau^{-(t-r)}\GG)'&=&(\Supp\FF)'\cap (\Supp\tau^{-(t-r)}\GG)'\\
                                                          &=& (\Supp\tau^{-(q-1-r)}\GG)'\\
                                                          &=&\{p+q-r-2\}.
      \end{array} \]
      Furthermore, $\Hom_A(\tau^{t}I_{r+1},\tau^{t-r}I_{p+q-r-2})\cong \Hom_A(\tau^{r}I_{r+1},I_{p+q-r-2})=0,$
      because, accordingly with  Proposition \ref{preinjecte} (5), $\tau^{r}I_{r+1}$ has no $S_{p+q-r-2}$
      as a composition factor. Finally, by Lemma \ref{repete},  we have that  $$\Ext^{1}_A(\tau^{t}I_{r+1},\tau^{t-r}I_{p+q-r-2})=0.$$
			
      \noindent(6) Let  $t\geq 1$ such that $t\equiv r\,\,(\mod\,q)$, $ r \in [1,q-2]$ and  $t\equiv p-1\,\,(\mod\,p)$.
      To complete the $s.s.$ $(F,G,\tau^{t}I_{p+r})$ we consider two situations:
      \begin{itemize}
      \item If $r<p$.  Let $M \cong \tau^{t-r}I_{p-(r+1)}$. We have that 
       \[\begin{array}{rcl}
      (\Supp\tau^{-(t-r)}\FF)'\cap(\Supp\tau^{-(t-r)}\GG)'&=&(\Supp\tau^{-(p-1-r)}\FF)'\cap (\Supp \GG)'\\
                                                          &=& (\Supp\tau^{(r+1)}\FF)'\\
                                                          &=&\{p-(r+1)\}.\\
      \end{array} \]
      Since $[\tau^{r}I_{p+r}:S_{p-(r+1)}]=0$, by Proposition \ref{fcpi}  (5), then 
      $$\Hom_A(\tau^{t}I_{p+r},\tau^{t-r}I_{p-(r+1)})\cong \Hom_A(\tau^{r}I_{p+r},I_{p-(r+1)})=0.$$
      By Lemma  \ref{repete}, we have that $\Ext^{1}_A(\tau^{t}I_{p+r},\tau^{t-r}I_{p-(r+1)})=0.$ 
      
      \item If $r \geq p$. Let $M\cong \tau^{t-(p-1)}I_{r}$. Then we have that
      \[\begin{array}{rcl}
      (\Supp\tau^{-[t-(p-1)]}\FF)'\cap(\Supp\tau^{-[t-(p-1)]}\GG)'&=&(\Supp\tau^{-[r-(p-1)]}\GG)'\\
                                                          &=& \{p+(r-p+1)-1 \}\\
                                                          &=&\{r\}. \\
      \end{array} \] 
      According to Proposition \ref{fcpi} (5), we have that  $$\Hom_A(\tau^{t}I_{p+r},\tau^{t-(p-1)}I_{r})\cong \Hom_A(\tau^{p-1}I_{p+r},I_{r}) \cong 
0$$
      and by Lemma \ref{repete} we have  that $\Ext^{1}_A(\tau^{t}I_{p+r},\tau^{t-(p-1)}I_{r})=0.$ 
      
     \end{itemize}
\end{proof}

Finally, Proposition \ref{completo1} and Proposition \ref{completo2} together are the main result of this paper, which establish the complete list 
of $c.s.s.$ of the form $(X,F,G,Y)$. 

\begin{thm} Let $A=K(\widetilde{\mathbb{A}}_{p,q})$. The  complete list of $c.s.s.$ of the form
$(M,F,G,Y)$  is as follows:
\begin{enumerate}
\item $(S_{p+q-1},F,G, P_0)$.
\item $(M,F,G, P_{p+q-1})$, where
      \begin{displaymath}
      \begin {array}{rll}
      M \cong &\left \{ \begin {array}{ll}
      \tau^{-p+1}P_{0}, \textrm{ if } p=q\\
      \tau^{-p+1}P_{q-1}, \textrm{ if } p \not =q.\\
      \end{array} \right.\\
      \end{array}
      \end{displaymath}
\item $(M,F,G,\tau^{-t} P_{p+q-1})$ with $t\geq 1$ such that $p\mid t$ and  $q\mid t$, where
\begin{displaymath}
      \begin {array}{rll}
      M \cong &\left \{ \begin {array}{ll}
      \tau^{-t-p+1}P_{0}, \textrm{ if } p=q\\
      \tau^{-t-p+1}P_{q-1}, \textrm{ if } p \not =q.\\
      \end{array} \right.\\
      \end{array}
\end{displaymath}
\item $(\tau^{-t+1}P_{p+q-1},F,G,\tau^{-t}P_0)$, with $t\geq 1$
      such that $p\mid t$ and  $q \mid t$.
\item $(\tau^{-t-(p-r-1)}P_{q+r-1},F,G,\tau^{-t} P_{p-r})$, with $t\geq 1$ such that $q\mid t$,  \\ $t\equiv r \,(\mod \,p)$ and $r \in [1,p-1]$.

    \item $(M, F,G,\tau^{-t}P_{p+q-r-1})$, with $t\geq 1$ such that $p\mid t$,  $t\equiv r \,(\mod \,q)$ and  $r \in [1, q-1]$, where
      \begin{displaymath}
      \begin {array}{rll}
      M \cong &\left \{ \begin {array}{ll}
      \tau^{-t-q+r+1}P_{p-q+r}, \textrm{ if } p \geq q-r\\
      \tau^{-t-p+1}P_{q-r-1}, \textrm{ if } p<q-r.\\
      \end{array} \right.\\
      \end{array}
\end{displaymath}
\item $(\tau^{t}I_{p-1},F,G,\tau^{t}I_{p})$, with $t\geq 1$ such that $t\equiv p-1\,\,(\mod\,p)$ and $q\mid t$.
\item $(\tau^{t}I_{p+q-2},F,G,\tau^{t}I_1)$, with $t\geq 1$ such that $t\equiv q-1\,\,(\mod\,q)$ and $p \mid t$. 
\item $(\tau^{t+1}I_{p+q-1},F,G,\tau^{t}I_{0})$, with $t\geq 1$ such that $t\equiv p-1 \,\,(\mod\,p)$ and $t\equiv  q-1\,\,(\mod\,q)$.
\item $(\tau^{t-p+1}I_{q-1},F,G,\tau^{t}I_{p+q-1})$, with $t\geq 1$ such that $t\equiv p-1\,\,(\mod\,p)$ and $t\equiv  q-1 \,\,(\mod \,q)$.
\item $(\tau^{t-r} I_{p+q-r-2},F,G,\tau^{t}I_{r+1})$, with $t\geq 1$ such that $t\equiv r\,\,(\mod\,p)$,
      $r \in [1,p-2]$ and $t\equiv q-1\,\,(\mod\,q)$.
\item $(M,F,G,\tau^{t}I_{p+r})$, with $t\geq 1$ such that $t\equiv r\,\,(\mod\,q)$, $r \in [1,q-2]$ and $t\equiv p-1\,(\mod\,p)$, where 
      \begin{displaymath}
      \begin {array}{rll}
      M \cong &\left \{ \begin {array}{ll}
      \tau^{t-r}I_{p-(r+1)}, \textrm{ if } r < p\\
      \tau^{t-(p-1)}I_{r}, \textrm{ if } r \geq p.\\
      \end{array} \right.\\
      \end{array}
      \end{displaymath}
\end{enumerate}

\end{thm}

\section*{Acknowledgments}
The first named author thanks CAPES and CNPq (235081/2014-0) for financial support. The second author thanks CNPq for the research grant, and also Fapesp for the 
support, processo tem\'atico 2014/-8608-0.
The authors thank Maria Izabel Ramalho Martins for helpful discussions. We also thank the referee for his suggestions, which certainly made this paper
more readable, and also for some useful corrections on the Mathematics. 

\thebibliography{99}{

\bibitem{ASS1}I. Assem, D. Simson and A. Skowronski. {\it Elements of Representation
  Theory of Associative Algebras. Vol 1: Techniques of Representation
  Theory}, Series: London Math. So. Student  Texts 65. Cambridge
  Univ. Press. (2006).
\bibitem{CMa}  P. Cadavid and E. do N.  Marcos, Stratifying systems over hereditary algebras,  {\it J. Algebra Appl.}, {\bf 14}, (2015), 1550093.

\bibitem{teseC}P. Cadavid, Sistemas estratificantes sobre álgebras hereditárias. Ph.D thesis. IME-USP. 2012.

\bibitem{CB} W.  Crawley-Boevey,  Exceptional sequences of representations of quivers.  In Representations of algebras, Proc. Ottawa 1992, eds V. Dlab 
and H. Lenzing, 
Canadian Math. Soc. Conf. Proc. 14 (Amer. Math. Soc., 1993), 117-124.

\bibitem{ES}K.  Erdman  and C.  Sáenz,   On standardly stratified algebras, {\it Comm.  Algebra}, {\bf 31}, (2003), 3429-3446.

\bibitem{MMS2} E. do N. Marcos, O. Mendoza and C. Sáenz,  Stratifying systems via relative simple mo\-du\-les, {\it J. Algebra}, {\bf 280}, (2004), 
472-487.

\bibitem{MMS1}  E. do N. Marcos, O. Mendoza and C. Sáenz, Stratifying systems via relative projective
mo\-du\-les, {\it  Comm. Algebra}, {\bf 33}, (2005), 1559-1573.
                      

\bibitem{SS2} D. Simson and A. Skowronski, {\it Elements of Representation Theory of Associative Algebras. Vol 2: Tubes and Conceladed Algebras of 
Euclidean type}, 
Series: London Math. Soc. Student Texts 71. Cambridge  Univ.  Press.
(2007).

}


\end{document}